\documentclass{amsart}
\usepackage{amssymb}
\usepackage[all]{xy}

\newtheorem{thm}{Theorem}[section]

\newtheorem{prop}[thm]{Proposition}
\newtheorem{cor}[thm]{Corollary}

\theoremstyle{definition}

\theoremstyle{remark}
\newtheorem{rmk}[thm]{Remark}
\newtheorem{exam}[thm]{Example}

\numberwithin{equation}{section}

\newcommand{\Pp}{\mathbb{P}}
\newcommand{\C}{\mathbb{C}}
\newcommand{\R}{\mathbb{R}}

\newcommand{\N}{\mathbb{N}}
\newcommand{\Q}{\mathbb{Q}}
\newcommand{\Z}{\mathbb{Z}}

\newcommand{\cG}{\mathcal{G}}
\newcommand{\cH}{\mathcal{H}}

\newcommand{\cM}{\mathcal{M}}

\newcommand{\cO}{\mathcal{O}}

\newcommand{\cT}{\mathcal{T}}

\newcommand{\isomto}{\overset{\sim}{\rightarrow}}

\newcommand{\ev}{{\textup{ev}}}

\DeclareMathOperator{\Ind}{Ind}

\DeclareMathOperator{\rank}{rk}
\DeclareMathOperator{\sech}{sech}
\newcommand{\rk}{\rank}

\title[The real Coxeter toric variety of type A]
{Rational cohomology of the real Coxeter toric variety of type A}
\author{Anthony Henderson}
\address{School of Mathematics and Statistics\\
University of Sydney NSW 2006\\
Australia}
\email{anthony.henderson@sydney.edu.au}
\thanks{The author's research was supported by the
Australian Research Council grant DP0985184.}

\begin{document}
\begin{abstract}
The toric variety corresponding to the Coxeter fan of type A can also be described as a De~Concini--Procesi wonderful model. Using a general result of Rains which relates cohomology of real De~Concini--Procesi models to poset homology, we give formulas for the Betti numbers of the real toric variety, and the symmetric group representations on the rational cohomologies. We also show that the rational cohomology ring is not generated in degree 1.
\end{abstract}
\maketitle
\section{Introduction}
Fix a positive integer $n$. Let $N$ denote the quotient $\Z^n/\Z(1,1,\cdots,1)$, which is a free $\Z$-module of rank $n-1$. We have a representation of the symmetric group $S_n$ on $N\otimes\R=\R^n/\R(1,1,\cdots,1)$ by permuting coordinates, which is generated by reflections in the hyperplanes 
\[ \{(a_1,a_2,\cdots,a_n)\in\R^n\,|\,a_i=a_j\}/\R(1,1,\cdots,1),\text{ for }
1\leq i\neq j\leq n. \]
Let $\Delta$ be the complete fan of rational strongly convex polyhedral cones in $N\otimes\R$ defined by this hyperplane arrangement (that is, the maximal cones of $\Delta$ are the closures of the chambers of the arrangement). Then to $N$ and $\Delta$ we can associate an $(n-1)$-dimensional nonsingular projective toric variety $\cT_n$ as in \cite{fulton}. In terms of simplicial complexes, $\cT_n$ is the toric variety associated to the Coxeter complex of $S_n$. Clearly $S_n$ acts on $\cT_n$ by variety automorphisms. 

This is the $W=S_n$ case of a construction which can be carried out for any Weyl group $W$, producing the Coxeter toric variety $\cT_W$. In \cite{procesi}, Procesi gave a formula for the rational cohomology $H^*(\cT_W,\Q)$ as a graded $\Q W$-module. Further results about the representation of $W$ on $H^*(\cT_W,\Q)$ were obtained by Stembridge in \cite{stembridge1,stembridge2} and Lehrer in \cite{lehrer}. These results make use of the well-developed general theory of the cohomology of complex toric varieties.

The rational cohomology of the real variety $\cT_W(\R)$, which is not covered by any such general theory, was considered by the author and Lehrer in \cite{hendersonlehrer}. The main result of that paper was a formula for the alternating sum $\sum_i (-1)^i H^i(\cT_W(\R),\Q)$ as a virtual $\Q W$-module. The methods used there did not allow us to isolate an individual $H^i(\cT_W(\R),\Q)$.

The present paper restricts to the case $W=S_n$. As Procesi observed in \cite[Section 3]{procesi}, there is an alternative construction of the variety $\cT_n$, as a special case of what he and De~Concini later called the wonderful model of a subspace arrangement. In the terminology of \cite{dp}, if $\cG$ denotes the building set in $(\C^n)^*$ which consists of all subspaces $\C\{x_i\,|\,i\in I\}$ where $I$ is a nonempty subset of $\{1,2,\cdots,n\}$ and $x_1,x_2,\cdots,x_n$ are the coordinate functions on $\C^n$, then the projective wonderful model $Y_n=\overline{Y}_\cG$ is isomorphic to $\cT_n$. We recall the definition of $Y_n$ and explain this isomorphism further in Section 2.

Now the rational cohomology of $\overline{Y}_\cG(\R)$ may be described using the result of Rains \cite[Theorem 3.7]{rains} (or rather its dual form, \cite[Theorem 2.1]{hendersonrains}). This description in general involves the homology of the poset $\Pi_\cG^{(2)}$, consisting of all direct sums of even-dimensional elements of $\cG$; in our case, this poset is clearly isomorphic to the poset of even-size subsets of $\{1,2,\cdots,n\}$, whose homology is easy to determine. 

In Section 3 we carry out the calculation and derive the following formula. Here we let $R(S_n)$ denote the Grothendieck group of $\Q S_n$-modules and define the $\N$-graded ring $\bigoplus_{n\geq 0} R(S_n)$ using the induction product, with identity element $1=1_{S_0}$ (the trivial representation of $S_0$). We then complete this to a power series ring $R=\widehat{\bigoplus}_{n\geq 0} R(S_n)$ and adjoin an indeterminate $t$.

\begin{thm} \label{mainthm}
We have the following equality in $R[t]$:
\[
1+\sum_{n\geq 1}\sum_i H^i(\cT_n(\R),\Q)\,(-t)^i
=(\sum_{n\geq 0}1_{S_n})(1+\sum_{\substack{n\geq 2\\n\textup{ even}}}\varepsilon_{S_n}t^{n/2})^{-1}.
\]
\end{thm}
\noindent
Note that setting $t=1$ in Theorem \ref{mainthm} recovers \cite[(8)]{hendersonlehrer}.

An alternative way to express this formula is as follows.
\begin{cor} \label{indcor}
We have the following equality in $R(S_n)$:
\[
H^i(\cT_n(\R),\Q)=\sum_{\substack{n_1,n_2,\cdots,n_m\geq 2\\n_1,n_2,\cdots,n_m\textup{ even}\\n_1+n_2+\cdots+n_m=2i\leq n}}(-1)^{i+m}\Ind_{S_{n-2i}\times S_{n_1}\times\cdots\times S_{n_m}}^{S_n}(\varepsilon_{n_1,\cdots,n_m}),
\]
where $\varepsilon_{n_1,\cdots,n_m}$ is the linear character whose restriction to the $S_{n-2i}$ factor is trivial and whose restriction to each $S_{n_j}$ factor is the sign character.
\end{cor}
\noindent
Corollary \ref{indcor} answers the question posed after \cite[Theorem 6]{hendersonlehrer} in the negative.
Another consequence of Theorem \ref{mainthm} is a formula for the Betti numbers of $\cT_n(\R)$.
\begin{cor} \label{betticor}
We have
\[
\dim H^i(\cT_n(\R),\Q)=A_{2i}\binom{n}{2i},
\]
where $A_{2i}$ denotes the Euler secant number, i.e.\ the coefficient of $\frac{x^{2i}}{(2i)!}$ in the power series $\sec(x)$. 
In particular, $H^i(\cT_n(\R),\Q)=0$ if $i>\frac{n}{2}$.
\end{cor}

From the De~Concini--Procesi description one sees that there is a natural morphism $\overline{\cM_{0,n+2}}\to\cT_n$, where $\overline{\cM_{0,n+2}}$ denotes the moduli space of stable genus $0$ curves with $n+2$ marked points. In Section 4, we use this morphism and the results of Etingof et al.\ \cite{etingofetal} on the rational cohomology of $\overline{\cM_{0,n+2}}(\R)$ to show that the cup products of elements of $H^1(\cT_n(\R),\Q)$ do not span $H^2(\cT_n(\R),\Q)$, for $n\geq 4$. Thus the cohomology ring $H^*(\cT_n(\R),\Q)$ is not generated in degree $1$, in contrast to \cite[Theorem 2.9]{etingofetal}. It would be interesting to find a presentation for it.
\section{The De~Concini--Procesi model}
Let $V$ be a finite-dimensional complex vector space, and $\cG$ a collection of nonzero subspaces of the dual space $V^*$, including $V^*$ itself, which satisfies the definition \cite[Section 2.3]{dp} of a building set. For any $A\in\cG$, let $A^\perp$ denote the orthogonal subspace of $V$. Let $\cM_\cG$ be the complement in $\Pp(V)$ of the union of all $\Pp(A^\perp)$ for $A\in\cG$, and let
\[ \rho:\cM_\cG\to\prod_{A\in\cG}\Pp(V/A^\perp) \]
be the morphism whose component in the factor labelled by $A$ is the restriction to $\cM_\cG$ of the obvious map $\Pp(V)\setminus\Pp(A^\perp)\to\Pp(V/A^\perp)$. Since the factor labelled by $A=V^*$ is $\Pp(V)$ itself, $\rho$ is injective. The projective De~Concini--Procesi model $\overline{Y}_\cG$ is defined in \cite[Section 4.1]{dp} to be the closure of $\rho(\cM_\cG)$ in $\prod_{A\in\cG}\Pp(V/A^\perp)$.

In this paper, we take $V=\C^n$ for a positive integer $n$, let $x_1,x_2,\cdots,x_n$ be the coordinate functions in $(\C^n)^*$, and set
\begin{equation} \label{cgeqn}
\cG=\{\C\{x_i\,|\,i\in I\}\,|\,\emptyset\neq I\subseteq [n]\}.
\end{equation}
Here $[n]$ denotes $\{1,2,\cdots,n\}$.
Then
\begin{equation}
\cM_\cG=\{[a_1:a_2:\cdots:a_n]\in\Pp^{n-1}\,|\,a_i\neq 0\text{ for all }i\}=T_n\text{ say,}
\end{equation}
the complement of the coordinate hyperplanes in $\Pp^{n-1}$. (Note that $\cG$ is the maximal building set for this particular complement.) We may obviously identify $T_n$ with the $(n-1)$-dimensional algebraic torus with cocharacter lattice $N=\Z^n/\Z(1,1,\cdots,1)$.

For any nonempty subset $I\subseteq[n]$, we set
\[
\Pp^I=\Pp(\C^n/\{(a_1,\cdots,a_n)\in\C^n\,|\,a_i=0\text{ for all }i\in I\})=\{[a_i^I]_{i\in I}\},
\]
where $[a_i^I]_{i\in I}$ denotes an $I$-tuple of homogeneous coordinates. (We identify $\Pp^{[n]}$ with $\Pp^{n-1}$ and drop the superscript $[n]$ in $a_i^{[n]}$.) Then the morphism $\rho$ becomes
\begin{equation}
\rho:T_n\to\prod_{\emptyset\neq I\subseteq [n]}\Pp^I,
\end{equation}
where the $I$-component of $\rho([a_1:a_2:\cdots:a_n])$ is $[a_i]_{i\in I}$. The De~Concini--Procesi model $Y_n=\overline{Y}_\cG$ is the closure of $\rho(T_n)$ as above. Notice that the factors $\Pp^I$ where $|I|=1$ are points, and make no difference to the definition of $Y_n$, but it is sometimes notationally convenient to keep them.

\begin{exam}
Clearly $Y_1\cong\Pp^0$ (a point), and $Y_2\cong\Pp^1$. 
\end{exam}
\begin{exam}
When $n=3$, we can identify $\rho$ with the map
\[
T_3\to\Pp^2\times\Pp^1\times\Pp^1\times\Pp^1:[a_1:a_2:a_3]\mapsto([a_1:a_2:a_3],[a_1:a_2],[a_1:a_3],[a_2:a_3]),
\]
where the $\Pp^1$ factors are, respectively, $\Pp^{\{1,2\}}$, $\Pp^{\{1,3\}}$, and $\Pp^{\{2,3\}}$. Thus $Y_3$ is the blow-up of $\Pp^2$ at the three points $[0:0:1]$, $[0:1:0]$, and $[1:0:0]$. 
\end{exam}
\begin{rmk}
For $n\geq 4$, $Y_n$ can be obtained by an iterated blow-up procedure: first one blows up $\Pp^{n-1}$ at the $n$ coordinate points, then one blows up again along the proper transforms of the $\binom{n}{2}$ coordinate lines, then one blows up again along the proper transforms of the $\binom{n}{3}$ coordinate planes, and so forth. This is the description given in \cite[Section 3]{procesi}. 
\end{rmk}

Note that the symmetric group $S_n$ acts on $\displaystyle\prod_{\emptyset\neq I\subseteq [n]}\Pp^I$ in a natural way: the action of $w\in S_n$ uses the isomorphisms 
\begin{equation} 
\Pp^I\isomto\Pp^{w(I)}:[a_i^I]_{i\in I}\mapsto[a_{w^{-1}(j)}^I]_{j\in w(I)}. 
\end{equation} 
Clearly $\rho$ is $S_n$-equivariant, where $S_n$ acts on $T_n$ by permuting coordinates. Thus $S_n$ acts on $Y_n$.
There is also a natural action of the torus $T_n$ on each $\Pp^I$ by 
\begin{equation}
[a_1:a_2:\cdots:a_n].[a_i^I]_{i\in I}=[a_ia_i^I]_{i\in I}. 
\end{equation} 
Clearly $\rho$ is $T_n$-equivariant, so $T_n$ acts on $Y_n$.

\begin{prop} \label{equationprop}
$Y_n$ is the closed subvariety of $\displaystyle\prod_{\emptyset\neq I\subseteq [n]}\Pp^I$ defined by the condition that $(a_i^I)_{i\in I}$ and $(a_i^J)_{i\in I}$ are linearly dependent for all $\emptyset\neq I\subseteq J\subseteq[n]$.
\end{prop}
\noindent
Note that $[a_j^J]_{j\in J}\in\Pp^J$ is a $J$-tuple of homogeneous coordinates, so by definition $a_j^J\neq 0$ for some $j\in J$. However, it is possible that $a_i^J=0$ for all $i$ in the smaller subset $I$, in which case the condition in Proposition \ref{equationprop} is automatically satisfied; otherwise, the condition is equivalent to saying that $[a_i^I]_{i\in I}=[a_i^J]_{i\in I}$.
\begin{exam} \label{y3exam}
When $n=3$, Proposition \ref{equationprop} asserts that
\[
([a_1:a_2:a_3],[b_1:b_2],[c_1:c_3],[d_2:d_3])\in\Pp^2\times\Pp^1\times\Pp^1\times\Pp^1
\] 
lies in $Y_3$ exactly when $a_1b_2=a_2b_1$, $a_1c_3=a_3c_1$, and $a_2d_3=a_3d_2$.
\end{exam}
\begin{proof}
It is clear that the condition in Proposition \ref{equationprop} can be rephrased in terms of equations in the coordinates $a_i^I$ as in Example \ref{y3exam}, so it does define a closed subvariety $Z$ of $\displaystyle\prod_{\emptyset\neq I\subseteq [n]}\Pp^I$. It is also clear that $\rho(T_n)\subseteq Z$. Hence $Y_n\subseteq Z$. To prove the reverse inclusion, suppose that $p\in Z$ has $I$-component $[a_i^I]_{i\in I}$. Let $K_1=[n]$, and define a chain of subsets $K_1\supset K_2 \supset K_3\supset\cdots\supset K_m\supset K_{m+1}=\emptyset$ by the recursive rule
\begin{equation} \label{vanishingeqn}
K_{\ell+1}=\{k\in K_\ell\,|\,a_k^{K_\ell}=0\}\text{ for }1\leq\ell\leq m,
\end{equation}
where $m$ is minimal such that $a_k^{K_m}\neq 0$ for all $k\in K_m$. For any nonempty $I\subseteq[n]$, there is some $s\leq m$ such that $I\subseteq K_s$, $I\not\subseteq K_{s+1}$. By definition of $Z$, we must have $[a_i^I]_{i\in I}=[a_i^{K_s}]_{i\in I}$. So $p$ is determined by its $K_s$-components for $1\leq s\leq m$, which are constrained only by the vanishing conditions \eqref{vanishingeqn}. For any $t\in\C^\times$, define $q(t)=[a_1(t):a_2(t):\cdots:a_n(t)]\in T_n$ by the rule that $a_i(t)=t^s a_i^{K_s}$, if $i\in K_s$ and $i\notin K_{s+1}$. Then it is easy to see that $\lim_{t\to 0}\rho(q(t))=p$, in the sense that we have a morphism 
\[ \sigma:\C\to Z:t\mapsto\begin{cases}
\rho(q(t)),&\text{ if }t\neq 0,\\
p,&\text{ if }t=0.
\end{cases} \]
This proves that $p\in\overline{\rho(T_n)}=Y_n$.
\end{proof}

We can now decompose $Y_n$ as the disjoint union of locally closed subvarieties $\cO_{K_1,\cdots,K_{m+1}}$, one for each chain of subsets $[n]=K_1\supset K_2\supset\cdots\supset K_{m+1}=\emptyset$. Namely, a point of $Y_n$ with $I$-component $[a_i^I]_{i\in I}$ lies in $\cO_{K_1,\cdots,K_{m+1}}$ if and only if $K_1,K_2,\cdots,K_{m+1}$ are the subsets defined recursively by \eqref{vanishingeqn}. 
\begin{exam}
For a point of $Y_3$ with notation as in Example \ref{y3exam}, the conditions for belonging to certain pieces of the disjoint union are as follows:
\[
\begin{split}
\cO_{[3],\emptyset}&:\ a_1,a_2,a_3\neq 0,\\
\cO_{[3],\{1\},\emptyset}&:\ a_1=0,\,a_2,a_3\neq 0,\\
\cO_{[3],\{1,2\},\emptyset}&:\ a_1=a_2=0,\ b_1,b_2\neq 0,\\
\cO_{[3],\{1,2\},\{1\},\emptyset}&:\ a_1=a_2=0,\ b_1=0.\\
\end{split}
\]
\end{exam}
\begin{prop} \label{orbitprop}
The subvarieties $\cO_{K_1,\cdots,K_{m+1}}$ are exactly the $T_n$-orbits in $Y_n$.
\end{prop}
\begin{proof}
It is clear that each $\cO_{K_1,\cdots,K_{m+1}}$ is $T_n$-stable.
As observed in the proof of Proposition \ref{equationprop}, a point of $\cO_{K_1,\cdots,K_{m+1}}$ is determined by its $K_s$-components, and the coordinates of the $K_s$-component indexed by $K_{s+1}$ must vanish by definition. From this it is immediate that $T_n$ acts transitively on $\cO_{K_1,\cdots,K_{m+1}}$.
\end{proof} 
For example, $\cO_{[n],\emptyset}=\rho(T_n)$. In general, 
\begin{equation}
\cO_{K_1,\cdots,K_{m+1}}\cong T_{|K_1|-|K_2|}\times T_{|K_2|-|K_3|}\times\cdots\times T_{|K_m|-|K_{m+1}|} 
\end{equation}
is a torus of dimension $n-m$. 

\begin{rmk}
The chains $[n]=K_1\supset K_2\supset\cdots\supset K_{m+1}=\emptyset$ are in obvious bijection with the nested subsets of $\cG$, in the terminology of \cite{dp}.
By an argument similar to the proof of Proposition \ref{equationprop}, one can show that the closure $\overline{\cO_{K_1,\cdots,K_{m+1}}}$ is the subvariety of $Y_n$ defined by the condition that $a_k^{K_\ell}=0$ for $k\in K_{\ell+1}$, for $1\leq\ell\leq m$. This is the closed subvariety $D_{K_1,\cdots,K_{m+1}}$ of \cite[Section 4.3]{dp}. It is clear that $D_{K_1,\cdots,K_{m+1}}$ is the union of all orbits $\cO_{L_1,\cdots,L_{m'+1}}$ such that the chain $(L_s)$ refines the chain $(K_s)$. We have an obvious isomorphism
\begin{equation}
D_{K_1,\cdots,K_{m+1}}\cong Y_{|K_1|-|K_2|}\times Y_{|K_2|-|K_3|}\times\cdots\times Y_{|K_m|-|K_{m+1}|}, 
\end{equation}
which is merely what \cite[Theorem 4.3]{dp} says for our special $\cG$.
\end{rmk}

Now let $\cT_n$ be the toric variety associated to the lattice $N=\Z^n/\Z(1,1,\cdots,1)$ and the Coxeter fan $\Delta$, as in the introduction. The torus which naturally acts on $\cT_n$ is the one with cocharacter lattice $N$, in other words the torus $T_n$. By \cite[Section 3.1]{fulton}, the $T_n$-orbits in $\cT_n$ are naturally in bijection with the cones of $\Delta$, and these in turn are in bijection with the chains $[n]=K_1\supset K_2\supset\cdots\supset K_{m+1}=\emptyset$. So the above description of $T_n$-orbits in $Y_n$ immediately suggests the following result.

\begin{prop}[De~Concini--Procesi] \label{isomprop}
There is an isomorphism $\cT_n\isomto Y_n$ which respects the actions of $S_n$ and $T_n$ on both varieties.
\end{prop}
\begin{proof}
The proof is omitted from \cite[Section 3]{procesi}, but an argument similar to that in \cite[Section IV]{dps} works. Namely, $Y_n$ is nonsingular by \cite[Theorem 4.2]{dp} and contains an open subvariety $\rho(T_n)$ on which $T_n$ acts simply transitively, so $Y_n$ is isomorphic to the toric variety associated to $N$ and some complete fan $\Delta'$ in $N\otimes\R$. Since $S_n$ acts on $Y_n$ in a way compatible with its action on $T_n$, the fan $\Delta'$ must be $S_n$-stable. The maximal cones in $\Delta'$ correspond to the $T_n$-fixed points in $Y_n$, which are exactly the $0$-dimensional orbits $\cO_{K_1,\cdots,K_{n+1}}$, where $|K_i|=n+1-i$ for all $i$. Hence $S_n$ permutes the maximal cones of $\Delta'$ simply transitively, which means that no reflecting hyperplane for the action of $S_n$ can intersect the interior of a maximal cone. Hence the interior of each maximal cone of $\Delta'$ is contained in a single chamber of the hyperplane complement, forcing $\Delta'=\Delta$.
\end{proof}
\section{Poset homology and the proof of Theorem \ref{mainthm}}
It is clear that the isomorphism of Proposition \ref{isomprop} respects the obvious real structures of $\cT_n$ and $Y_n$, so the real variety $\cT_n(\R)$ is isomorphic to $Y_n(\R)$. This allows us to use Rains' theorem \cite[Theorem 3.7]{rains}, or rather its dual form \cite[Theorem 2.1]{hendersonrains}, which expresses the rational cohomology of $\overline{Y}_\cG(\R)$ in terms of poset homology. By \cite[Theorem 4.1]{gaiffi}, when $\cG$ is the complexification of a real building set as in the present case, there is no difference between the real locus of the complex variety $\overline{Y}_\cG$ and the real variety as defined in \cite{rains}. 

The general statement of Rains' theorem involves the poset $\Pi_\cG^{(2)}$ of direct sums of even-dimensional elements of $\cG$. In our case, this is clearly the same as the poset of even-dimensional elements of $\cG$ with the minimal element $\{0\}$ added, which in turn is isomorphic to the poset $B_n^\ev$ of even-size subsets of $[n]$. Hence \cite[Theorem 2.1]{hendersonrains} gives us an isomorphism of $\Q S_n$-modules
\begin{equation} \label{firstrainseqn}
H^i(Y_n(\R),\Q)\cong\bigoplus_{I\in B_n^\ev} H_{|I|-i}(\emptyset,I)\otimes\varepsilon_I.
\end{equation}
Here $H_{m}(\emptyset,I)$ denotes a poset homology group with $\Q$-coefficients of the open interval $(\emptyset,I)$ in the poset $B_n^\ev$. Our degree convention follows \cite{henderson2, hendersonrains}: thus $H_m(\emptyset,I)$ is what would be called $\widetilde{H}_{m-2}((\emptyset,I),\Q)$ in texts such as \cite{wachs}. By definition, $H_m(\emptyset,\emptyset)$ is $\Q$ if $m=0$ and $0$ otherwise, and if $I$ is an atom of $B_n^\ev$ (that is, $|I|=2$), then $H_m(\emptyset,I)$ is $\Q$ if $m=1$ and $0$ otherwise. The tensoring with $\varepsilon_I$ also requires explanation: this is to indicate that if $w\in S_n$ preserves $I$, we want not its usual action on the homology group $H_{|I|-i}(\emptyset,I)$, but that action multiplied by $\varepsilon(w)$.

Now $B_n^\ev$ is a rank-selected subposet of the Boolean lattice $B_n$, and is itself a ranked poset with $\rk(I)=|I|/2$. By \cite[Theorem 3.4.1]{wachs}, $B_n^\ev$ is Cohen--Macaulay. In particular, $H_m(\emptyset,I)$ is nonzero if and only if $m=\rk(I)$, or in other words $|I|=2m$. So \eqref{firstrainseqn} can be simplified to
\begin{equation} \label{secondrainseqn}
H^i(Y_n(\R),\Q)\cong\bigoplus_{\substack{I\in B_n^\ev\\|I|=2i}} H_{i}(\emptyset,I)\otimes\varepsilon_I.
\end{equation}  
Here the right-hand side is just a sign-twisted version of the Whitney homology
\[
WH_i(B_n^\ev)=\bigoplus_{\substack{I\in B_n^\ev\\|I|=2i}} H_{i}(\emptyset,I).
\]
Note that if $|I|=2i$, then the closed interval $[\emptyset,I]$ is isomorphic to $B_{2i}^\ev$.

The homology of $B_n^\ev$ is expressed as a representation of $S_n$ by \cite[Theorem 3.4.4]{wachs}. Here we derive a different expression via the method of \cite[Section 4.4]{wachs} and \cite{henderson2}, using the ring $R=\widehat{\bigoplus}_{n\geq 0}R(S_n)$ from the introduction.
\begin{prop} \label{schprop}
We have the following equality in $R$:
\[
1+\sum_{\substack{n\geq 2\\n\textup{ even}}}(-1)^{n/2}H_{n/2}(\emptyset,[n])=(1+\sum_{\substack{n\geq 2\\n\textup{ even}}}1_{S_n})^{-1}.
\]
\end{prop}
\begin{proof}
If we multiply $1+\displaystyle\sum_{\substack{n\geq 2\\n\textup{ even}}}(-1)^{n/2}H_{n/2}(\emptyset,[n])$ and $1+\displaystyle\sum_{\substack{n\geq 2\\n\textup{ even}}}1_{S_n}$ in $R$, the result clearly has terms only in even degrees. We must show that if $n\geq 2$ is even, then the degree-$n$ term of this product vanishes. By definition, this term is
\[
\sum_{\substack{0\leq m\leq n\\m\textup{ even}}}(-1)^{m/2}\Ind_{S_m\times S_{n-m}}^{S_n}(H_{m/2}(\emptyset,[m])\boxtimes 1_{S_{n-m}}),
\]
which is clearly equal to
\[
\sum_{\substack{0\leq m\leq n\\m\textup{ even}}}(-1)^{m/2}\bigoplus_{\substack{I\in B_n^\ev\\|I|=m}} H_{m/2}(\emptyset,I)=\sum_{\substack{0\leq m\leq n\\m\textup{ even}}}(-1)^{m/2}\, WH_{m/2}(B_n^\ev).
\]
This vanishes by a well-known general principle, for which one reference is \cite[Theorem 4.4.1]{wachs} (see also \cite[Theorem 4.1]{henderson2}). 
\end{proof}  

We can now give the proof of Theorem \ref{mainthm}.
\begin{proof}
Proposition \ref{schprop} implies the following equality in $R[t]$:
\begin{equation}
(1+\sum_{\substack{n\geq 2\\n\textup{ even}}}\varepsilon_{S_n}t^{n/2})^{-1}=
\sum_{\substack{n\geq 0\\n\textup{ even}}}(H_{n/2}(\emptyset,[n])\otimes\varepsilon_{S_n})\,(-t)^{n/2}.
\end{equation}
Hence for $n\geq 1$, the degree-$n$ term of the right-hand side of Theorem \ref{mainthm} equals
\[
\sum_{\substack{0\leq m\leq n\\m\textup{ even}}}\Ind_{S_m\times S_{n-m}}^{S_n}((H_{m/2}(\emptyset,[m])\otimes\varepsilon_{S_m})\boxtimes 1_{S_{n-m}})\,(-t)^{m/2},
\]
which in turn equals
\[
\sum_{\substack{0\leq m\leq n\\m\textup{ even}}}(\bigoplus_{\substack{I\in B_n^\ev\\|I|=m}}H_{m/2}(\emptyset,I)\otimes\varepsilon_{I})\,(-t)^{m/2}.
\]
By \eqref{secondrainseqn} and Proposition \ref{isomprop}, this equals the degree-$n$ term of the left-hand side of Theorem \ref{mainthm}.
\end{proof}

To deduce Corollary \ref{indcor} from Theorem \ref{mainthm}, we simply observe that
\begin{equation*}
\begin{split}
&\qquad(1+\sum_{\substack{n\geq 2\\n\textup{ even}}}\varepsilon_{S_n}t^{n/2})^{-1}\\
&=1+\sum_{m\geq 1}(-1)^m(\sum_{\substack{n\geq 2\\n\textup{ even}}}\varepsilon_{S_n}t^{n/2})^m\\
&=1+\sum_{m\geq 1}(-1)^m\sum_{\substack{n_1,n_2,\cdots,n_m\geq 2\\n_1,n_2,\cdots,n_m\textup{ even}}}\Ind_{S_{n_1}\times\cdots\times S_{n_m}}^{S_{n_1+\cdots+n_m}}(\varepsilon_{S_{n_1}}\boxtimes\cdots\boxtimes\varepsilon_{S_{n_m}})\,t^{(n_1+\cdots+n_m)/2}.
\end{split}
\end{equation*}

Finally, if we apply to Theorem \ref{mainthm} the ring homomorphism $R\to\Q[\![x]\!]$ which sends $V\in R(S_n)$ to $(\dim V)\frac{x^n}{n!}$, we obtain
\begin{equation}
1+\sum_{n\geq 1}\sum_i \dim H^i(\cT_n(\R),\Q)\,(-t)^i\frac{x^n}{n!}
=\exp(x)\sech(t^{1/2}x),
\end{equation}
from which Corollary \ref{betticor} follows.
\section{Comparison with the moduli space of stable genus-zero curves}
In Section 2, we showed that $\cT_n$ is isomorphic to the De~Concini--Procesi model $\overline{Y}_\cG$ associated to the building set $\cG=\{\C\{x_i\,|\,i\in I\}\,|\,\emptyset\neq I\subseteq [n]\}$ in the dual of the vector space $V=\C^n$. A better-known example of the De~Concini--Procesi construction is the moduli space of stable genus-$0$ curves with marked points. 

Specifically, let $W=\C^{n+1}/\C(1,1,\cdots,1)$, and consider the building set 
\begin{equation}
\cH=\{\C\{y_j-y_{j'}\,|\,j,j'\in J\}\,|\,J\subseteq [n+1], |J|\geq 2\}\text{ in }W^*,
\end{equation}
where $y_1,\cdots,y_{n+1}$ are the coordinate functions on $\C^{n+1}$, so that $y_j-y_{j'}$ is a well-defined element of $W^*$. Then $\cM_\cH$ is the complement in $\Pp(W)$ of the hyperplanes $y_i=y_j$ generating the reflection representation of $S_{n+1}$. (Note that $\cH$ is the minimal building set for this particular complement.) One can identify $\cM_\cH$ with the moduli space $\cM_{0,n+2}$ of ordered configurations of $n+2$ points in $\Pp^1$, since the $(n+2)$th point can be assumed to be the point at infinity, leaving an $(n+1)$-tuple of distinct numbers modulo translation and scaling. 

As observed in \cite[Section 4.3, Remark (3)]{dp}, the De~Concini--Procesi model $\overline{Y}_\cH$ is isomorphic to the moduli space $\overline{\cM_{0,n+2}}$ of stable genus-$0$ curves with $n+2$ marked points, in such a way that the $S_{n+1}$-action on $\overline{Y}_\cH$ is identified with the restriction to $S_{n+1}$ of the $S_{n+2}$-action on $\overline{\cM_{0,n+2}}$. (Some of the details of this identification are in \cite{henderson1} and \cite{rains}.) By definition, we have a projection morphism from $\overline{Y}_\cH$ to
\[
\begin{split}
&\Pp(W/\C\{y_j-y_{j'}\,|\,j,j'\in J\}^\perp)\\
&=\Pp(\C^{n+1}/\{(a_1,\cdots,a_{n+1})\in\C^{n+1}\,|\,a_j=a_{j'}\text{ for all }j,j'\in J\})
\end{split}
\]
for all subsets $J\subseteq [n+1]$ with $|J|\geq 2$. A point $p$ in $\overline{Y}_\cH$ is determined by its images under these morphisms, which we will refer to as the $J$-components of $p$. 
\begin{exam}
When $|J|=3$, the $J$-component of $p$ may be thought of as a point of $\Pp^1\cong\overline{\cM_{0,4}}$. In the terminology of stable curves as found in \cite[Section 2]{etingofetal}, the resulting projection morphism $\overline{\cM_{0,n+2}}\to\overline{\cM_{0,4}}$ is the map of `stably forgetting' all marked points except those labelled by $J\cup\{n+2\}$.
\end{exam}

Let $f:W\to V$ be the $S_n$-equivariant vector space isomorphism whose transpose map $f^*:V^*\to W^*$ sends $x_i$ to $y_i-y_{n+1}$ for all $1\leq i\leq n$. If $A\in\cG$, then clearly $f^*(A)\in\cH$. Thus, in the terminology of Rains, $f:(W,\cH)\to(V,\cG)$ is a morphism of building sets. He observes in \cite[Proposition 2.4]{rains} that the De~Concini--Procesi construction is functorial, so we have a morphism $\overline{Y}_f:\overline{Y}_\cH\to\overline{Y}_\cG=Y_n$, which is surjective and birational by \cite[Proposition 2.7]{rains}. By definition, for any $p\in\overline{Y}_\cH$ and any nonempty subset $I\subseteq [n]$, the $I$-component of $\overline{Y}_f(p)$ is the image of the $(I\cup\{n+1\})$-component of $p$ under the map
\begin{equation} 
\begin{split}
&\C^{n+1}/\{(a_1,\cdots,a_{n+1})\in\C^{n+1}\,|\,a_i=a_{n+1}\text{ for all }i\in I\}\\
&\qquad\qquad\qquad\,\to \C^n/\{(a_1,\cdots,a_n)\in\C^n\,|\,a_i=0\text{ for all }i\in I\}:\\
&(a_1,\cdots,a_{n+1})\mapsto(a_1-a_{n+1},\cdots,a_n-a_{n+1}). 
\end{split}
\end{equation}
Composing the morphism $\overline{Y}_f$ with the isomorphism $\overline{\cM_{0,n+2}}\isomto\overline{Y}_\cH$, we obtain an $S_n$-equivariant surjective birational morphism $\tau:\overline{\cM_{0,n+2}}\to Y_n$, which respects the real structures. We now aim to use $\tau$ to compare these two varieties.

For any ordered $m$-element subset $\{s_1,s_2,\cdots,s_m\}$ of $[n]$, we have a morphism $\phi_{s_1,\cdots,s_m}:Y_n\to Y_m$, analogous to that considered in \cite[Section 2.3]{etingofetal}, which is defined as follows. For any $p\in Y_n$ with $I$-component $[a_i^I]_{i\in I}$ and any nonempty subset $J\subseteq [m]$, the $J$-component of $\phi_{s_1,\cdots,s_m}(p)$ is $[a_{s_j}^{S_J}]_{j\in J}$ where $S_j=\{s_j\,|\,j\in J\}$. It is clear from Proposition \ref{equationprop} that this does define a point of $Y_m$.

As a consequence, we have ring homomorphisms $\phi_{s_1,\cdots,s_m}^*:H^*(Y_m(\R),\Q)\to H^*(Y_n(\R),\Q)$. For any $i\neq j$ in $[n]$, let $\nu_{ij}$ be the image under $\phi_{i,j}^*:H^1(Y_2(\R),\Q)\to H^1(Y_n(\R),\Q)$ of the fundamental class of $Y_2(\R)\cong\Pp^1(\R)$. Since $\phi_{j,i}:Y_n\to\Pp^1$ is the composition of $\phi_{i,j}:Y_n\to\Pp^1$ with the inversion map on $\Pp^1$, we have $\nu_{ji}=-\nu_{ij}$. If $w\in S_n$, it is clear that $w.\nu_{ij}=\nu_{w(i)w(j)}$. 

\begin{prop} \label{h1prop}
The elements $\{\nu_{ij}\,|\,1\leq i<j\leq n\}$ form a basis of $H^1(Y_n(\R),\Q)$.
\end{prop}
\begin{proof}
Corollary \ref{betticor} says that $\dim H^1(Y_n(\R),\Q)=\binom{n}{2}$, so it suffices to show that the elements $\nu_{ij}$ for $i<j$ are linearly independent. For this, it is enough to show that their images $\tau^*(\nu_{ij})$ under the map $\tau^*:H^1(Y_n(\R),\Q)\to H^1(\overline{\cM_{0,n+2}}(\R),\Q)$ are linearly independent. But from the definitions, it follows that the composition $\phi_{i,j}\circ\tau:\overline{\cM_{0,n+2}}\to\Pp^1$ is precisely the map denoted $\phi_{i,j,n+1,n+2}$ in \cite[Section 2.3]{etingofetal}. Hence $\tau^*(\nu_{ij})$ is the element denoted $\omega_{i,j,n+1,n+2}$ there. Since the elements $\omega_{i,j,k,n+2}$ for $i<j<k<n+2$ form a basis of $H^1(\overline{\cM_{0,n+2}}(\R),\Q)$ by \cite[Proposition 2.3, Theorem 2.9]{etingofetal}, the desired linear independence holds.
\end{proof}

However, there is a major difference between the rational cohomology rings $H^*(\overline{\cM_{0,n+2}}(\R),\Q)$ and $H^*(Y_n(\R),\Q)$: the former is generated in degree $1$ by \cite[Theorem 2.9]{etingofetal}, whereas the latter is not (for $n\geq 4$), by the following result. 
\begin{prop} \label{nongenprop}
Assume $n\geq 4$. Let $C$ be the subspace of $H^2(Y_n(\R),\Q)$ spanned by the cup products of elements of $H^1(Y_n(\R),\Q)$. Then $C$ has basis 
\[
\{\nu_{ij}\nu_{kl},\nu_{ik}\nu_{jl},\nu_{il}\nu_{jk}\,|\,1\leq i<j<k<l\leq n\}.
\]
In particular, $\dim C=3\binom{n}{4}<5\binom{n}{4}=\dim H^2(Y_n(\R),\Q)$.
\end{prop}
\begin{proof}
Note first that $H^2(Y_3(\R),\Q)=0$ by Corollary \ref{betticor} (or since $Y_3(\R)$ is a non-orientable surface). So in $H^*(Y_3(\R),\Q)$ we have the relation $\nu_{12}\nu_{13}=0$. Applying the map $\phi_{i,j,k}^*$ to this relation, we deduce that $\nu_{ij}\nu_{ik}=0$ in $H^*(Y_n(\R),\Q)$ for any distinct $i,j,k$ in $[n]$ (and, of course, we also have $\nu_{ij}\nu_{ij}=0$ by skew-symmetry of the cup product). Hence $C$ is spanned by the given set. To show the linear independence, we can argue as in Proposition \ref{h1prop}: it is enough to show that
\[
\begin{split}
\{\omega_{i,j,n+1,n+2}\,\omega_{k,l,n+1,n+2},\ &\omega_{i,k,n+1,n+2}\,\omega_{j,l,n+1,n+2},\\
&\omega_{i,l,n+1,n+2}\,\omega_{j,k,n+1,n+2}\,|\,1\leq i<j<k<l\leq n\}
\end{split}
\]
is linearly independent in $H^2(\overline{\cM_{0,n+2}}(\R),\Q)$, and this follows from \cite[Proposition 2.3, Theorem 2.9]{etingofetal}. 
\end{proof}

In the case of $\overline{Y}_\cH\cong\overline{\cM_{0,n+2}}$, the $S_{n+1}$-action which arises naturally from the De~Concini--Procesi construction can be extended to an $S_{n+2}$-action. As a final note, we can deduce from Proposition \ref{nongenprop} that no such extension is possible for $Y_n$ (at least, in a way that respects the real structure).
\begin{cor}
For $n\geq 4$, the $S_n$-action on $Y_n(\R)$ cannot be extended to $S_{n+1}$. 
\end{cor}
\begin{proof}
It suffices to show that the representation of $S_n$ on the subspace $C$ in Proposition \ref{nongenprop} cannot be extended to $S_{n+1}$. We use the standard parametrization of irreducible $\Q S_n$-modules by partitions of $n$. It is immediate from the given basis that
$C\cong\Ind_{S_4\times S_{n-4}}^{S_n}(V_{(2,1,1)}\boxtimes 1_{S_{n-4}})$. Hence by the Pieri rule,
\begin{equation} \label{ceqn}
C\cong \begin{cases}
V_{(2,1,1)},&\text{ if $n=4$,}\\
V_{(3,1,1)}\oplus V_{(2,2,1)}\oplus V_{(2,1,1,1)},&\text{ if $n=5$,}\\
V_{(n-2,1,1)}\oplus V_{(n-3,2,1)}\oplus V_{(n-3,1,1,1)}\oplus V_{(n-4,2,1,1)},&\text{ if $n\geq 6$.}
\end{cases}
\end{equation}
Now the $\Q S_{n+1}$-module $V_\lambda$, for $\lambda$ a partition of $n+1$, restricts to give the $\Q S_n$-module $\bigoplus_\mu V_\mu$ where $\mu$ runs over all partitions of $n$ whose diagram is contained in that of $\lambda$. It is easy to see from \eqref{ceqn} that no sum of $\Q S_{n+1}$-modules $V_\lambda$ can possibly restrict to give $C$.
\end{proof}
\textbf{Acknowledgements.} This paper was conceived in June 2010 during the program `Configuration Spaces: Geometry, Combinatorics and Topology' at the Centro di Ricerca Matematica Ennio De Giorgi in Pisa. I thank the organizers for their generous support. I am also indebted to Gus Lehrer for helpful conversations about these matters.

\end{document}